\documentclass[review]{elsarticle}

\usepackage{hyperref}

\usepackage{amsthm}

\newtheorem{theorem}{Theorem}
\newtheorem{proposition}[theorem]{Proposition}
\newtheorem{corollary}[theorem]{Corollary}
\newtheorem{lemma}[theorem]{Lemma}

\begin{document}

\begin{frontmatter}

\title{The Overfull Conjecture on Split-Comparability Graphs}







\author[ufscar]{Jadder B. de Sousa Cruz}
\ead{bismarck.sc@gmail.com}
\author[ufscar]{C\^{a}ndida N. da Silva\corref{mycorrespondingauthor}}
\cortext[mycorrespondingauthor]{Corresponding author}
\ead{candida@ufscar.br}
\author[utfpr]{Sheila M. de Almeida}
\ead{sheilaalmeida@utfpr.edu.br}

\address[ufscar]{DComp-So -- \textsc{ccgt} -- \textsc{ufsc}ar -- Sorocaba, 
\textsc{sp},  Brazil}
\address[utfpr]{DAINF --  \textsc{utfpr}  --  Ponta Grossa, \textsc{pr}, Brazil}


\begin{abstract}
We show in this paper that a split-comparability graph $G$ has chromatic index equal to $\Delta(G) + 1$ if and only if $G$ is neighborhood-overfull. That implies the validity of the Overfull Conjecture for the class of split-comparability graphs. 
\end{abstract}


\end{frontmatter}

\section{Introduction}\label{intro}

Let $G = (V(G),E(G))$ be a simple and undirected graph with vertex set $V(G)$, edge set $E(G)$, $|V(G)| = n$ and $|E(G)| = m$. An \emph{edge coloring} is a mapping $c : E(G) \rightarrow C$, where $C$ is a set of colors and for any two adjacent edges $e$ and $f$,  $c(e) \neq c(f)$. When $|C|=k$, we say $c$ is a \emph{$k$-edge-coloring}. If for a vertex $v$ there is an edge incident with $v$ with color $c_1$, we say color $c_1$ is \emph{incident} with $v$, otherwise we say $c_1$ \emph{misses} $v$.  The \emph{chromatic index} of a graph $G$, denoted by $\chi'(G)$, is the smaller $k$ such that $G$ has a $k$-edge-coloring. The maximum degree of $G$ is denoted by $\Delta(G)$, and we say that a vertex with degree equal to $\Delta(G)$ is a \emph{$\Delta(G)$-vertex}. We can easily see that $\chi'(G) \ge \Delta(G)$. In 1964,  Vizing~\cite{vizi64} proved that the upper bound for $\chi'(G)$ is $\Delta(G) + 1$. The \emph{Classification Problem} is to determine for a given graph $G$ if $\chi'(G) =  \Delta(G)$ or $\chi'(G) =  \Delta(G) + 1$. If $\chi'(G) =  \Delta(G)$ then we say $G$ is \emph{Class 1}, otherwise, $G$ is \emph{Class 2}. The Classification Problem is $\mathcal{NP}$-Complete~\cite{holy81}.

A graph $G$ is \emph{overfull} if $n$ is odd and  $m > (n  - 1)\Delta(G)/2$. An overfull graph has therefore more edges than any $\Delta(G)$-edge-coloring can color, i.~e., it is necessarily Class 2. A graph $G$ is \emph{subgraph-overfull} if it has a subgraph $H$ with $\Delta(G) = \Delta(H)$ and $H$ is overfull. The \emph{neighborhood} of a vertex $v$, denoted by $N(v)$, is the set of vertices adjacent to $v$.  The \emph{closed neighborhood} of a vertex $v$ is the set $N[v] = N(v) \cup \{v\}$. Let $X$ be a set of vertices, then $N(X)$ is the union of $N(x)$ for every $x \in X$. Similarly, $N[X]$ is the union of $N[x]$ for every $x \in X$. The subgraph of $G$ \emph{induced} by $X \subset V(G)$, denoted by $G[X]$, is the graph whose vertex set is $X$ and whose edge set consists of all of the edges in $E(G)$ that have both endpoints in $X$. Similarly, the subgraph of $G$ \emph{induced} by $A \subset E(G)$, denoted by $G[A]$, is the graph whose vertex set consists of all of the endpoints of edges in $A$ and whose edge set is $A$. A graph $G$ is \emph{neighborhood-overfull} if for some $\Delta(G)$-vertex $v$, $G[N[v]]$ is overfull. Every subgraph-overfull or neighborhood-overfull graph must also be Class 2. Chetwynd and Hilton~\cite{chehil86} proposed the celebrated Overfull Conjecture, which states that a graph $G$ with $\Delta(G) > n/3$ is Class 2 if and only if it is subgraph-overfull.
The Overfull Conjecture remains open, but has been proved for several classes of graphs. Plantholt~\cite{plan81} proved it for graphs containing a \emph{universal vertex}, i.~e., a vertex adjacent to every other vertex of the graph. Hoffman and Rodger~\cite{horo92} proved it for \emph{complete multipartite graphs}, i.~e., graphs whose vertex set can be partitioned into subsets such that a pair of vertices is adjacent if and only if they belong to distinct subsets.  

A \emph{clique} is a set of vertices whose induced subgraph is complete, that is, every two distinct vertices in the clique are adjacent. A clique is \emph{maximal} if no other clique contains it. A \emph{stable set} is  a set of vertices whose induced subgraph is empty, that is, every two distinct vertices in the stable set are not adjacent. A graph $G$ is a \emph{split} graph if its vertices can be partitioned into two sets, say $Q$ and $S$, such that $Q$ is a clique and $S$ is a stable set. We assume $Q$ is maximal and  both $Q$ and $S$ are not empty throughout this paper. We denote by $B$ the bipartite subgraph induced by the edges with one endpoint in $Q$ and other endpoint in $S$. We denote by $d_Q$ the maximum degree of vertices of $Q$ in $B$ and $d_S$ the maximum degree of vertices of $S$ in $B$. 

The well structured subclasses of split graphs for which the Classification Problem is solved are the complete split graphs~\cite{chfuko95}, the split-indifference graphs~\cite{orti98} and all split graphs with odd $\Delta(G)$~\cite{chfuko95}. There are other partial results on the Classification Problem for split graphs; we refer the reader to the thesis~\cite{alme12} for a rather complete description of such partial results. Some of these results are used in the proof of Theorem~\ref{thm:principal}, and stated in Section~\ref{section:premi}.

A \emph{chordal} graph is a graph where each cycle with length at least 4 has at least one chord. Every split graph is a chordal graph~\cite{foha76}. Figueiredo, Meidanis and Mello~\cite{fimeme91} conjectured that every Class 2 chordal graph is neighborhood-overfull.
In this paper we solve the Classification Problem for a subclass of split graphs known as split-comparability graphs, defined in the next section. We also show that the main result in this paper implies that both the Overfull Conjecture and the Conjecture of Figueiredo, Meidanis and Mello hold for split-comparabilty graphs. The proof of our main result is presented in Section~\ref{section:proof} but, before that, we will present in Section~\ref{section:premi} some foundations that made this result possible.

\section{Preliminaries}\label{section:premi}


A \emph{transitive orientation} of a graph $G$ is an orientation of its edges such that whenever the directed edges $(x,y)$ and $(y,z)$ exist, then so does $(x,z)$. A graph $G$ is a \emph{comparability graph} if it admits a transitive orientation of its edges.  A \emph{split-comparability graph} is a split graph  that is also a comparability graph. The following theorem shows a characterization of split-comparability graphs.

\begin{theorem}~\cite{orvi96}
A split graph $G = (S \cup Q, E(G))$ is a split-comparability graph if and only if $Q$ can be totally ordered $v_1 < v_2 < \ldots < v_r$ and partitioned into three (possibly empty) segments $Q_l = \{v_1, v_2, \ldots, v_p\}$, $Q_r = \{v_q,v_{q+1},\ldots,v_r\}$ and $Q_t = Q \setminus (Q_l \cup Q_r)$ such that $N(s)$ has one of the following forms for every vertex $s \in S$: (i) $\{v_1,\ldots,v_i\}$, $i \leq p$; (ii) $\{v_j,\ldots,v_r\}$, $q \leq j \leq r$; (iii) $\{v_1,\ldots,v_i\} \cup \{v_j,\dots,v_r\}$, $i \leq p$ and $q \leq j \leq r$.
\label{thm:split-comp}
\end{theorem}

The sets with neighborhood with forms $(i)$, $(ii)$ e $(iii)$ are denoted by $S_l$, $S_r$ and $S_t$, respectively. We assume, without loss of generality, that $p = |N(S_l)|$ and $|N(S_r)| = r - q + 1$. In other words, the only vertices of $Q$ that are not incident with edges of $B$ are precisely those in $Q_t$.

A \emph{color class} in an edge coloring is the set of all edges with the same color. An edge coloring is \emph{balanced} if the cardinality of any two different color classes differ by at most 1.
A \emph{universal vertex} is a vertex adjacent to all the other vertices in the graph, so its degree is $n - 1$. The next results are well known.

\begin{lemma}~\cite{fofu69}
If $G$ has a $k$-edge-coloring then $G$ has a balanced $k$-edge-coloring.
\label{lem:colo_equi}
\end{lemma}
\begin{theorem}~\cite{plan81}
Let $G$ be a graph with a universal vertex, then $G$ is Class 2 if and only if $G$ is overfull.
\label{thm:vert-universal}
\end{theorem}
\begin{theorem}~\cite{chfuko95}
Let $G$ be a split graph. If $\Delta(G)$ is odd, then $G$ is Class 1.
\label{thm:chenfuko}
\end{theorem}


The \emph{core} of graph $G$, denoted by $\Lambda_G$, is the set of $\Delta(G)$-vertices. The   \emph{semicore} of $G$ is the subgraph induced by $N[\Lambda_G]$. We shall denote the semicore of $G$ by $K_G$. The next two theorems are more recent contributions to the theory of edge coloring.

\begin{theorem}~\cite{mafi10}
Let $G$ be a graph and $K_G$  its semicore. Then, $\chi'(G) = \chi'(K_G)$.
\label{thm:seminucleo}
\end{theorem}
\begin{theorem}~\cite{alme12}
Let $G = (S \cup Q, E(G))$ be a split graph. If it has a vertex $v \in S$ with $|Q|/2 \leq d(v) \leq \Delta(G) / 2$ then $G$ is Class 1.
\label{thm:sheila}
\end{theorem}


A graph $G$ is \emph{saturated} if $n$ is odd and $\overline{G}$ has exactly $\Delta(G)/2$ edges. In other words, $G$ has the maximum number of edges possible without being overfull, therefore $(n-1)\Delta(G)/2$ edges. Every saturated graph has a universal vertex. 
The following propositions are needed in the proof of our main theorem. Their proofs are omitted as these are straightforward consequences of the definitions of a balanced edge coloring and a split-comparability graph, respectively.

\begin{proposition}
Let $G$ be a saturated graph. Then, in any balanced edge coloring of $G$ with  $\Delta(G)$ colors, every color misses exactly one vertex.
\label{prop:satur_sobra}
\end{proposition}


\begin{proposition}
Let $G = (S \cup Q, E(G))$ be a split-comparability graph. There is at least one vertex in $Q_l$ ($Q_r$) adjacent to every vertex in $S_l \cup S_t$ ($S_r \cup S_t$).
\label{prop:qlsl}
\end{proposition}


\begin{corollary} 
Let $G= (S \cup Q, E(G))$ be a split-comparability graph. The subgraph induced by $Q \cup S_l \cup S_t$ ($Q \cup S_r \cup S_t$) has a universal vertex.
\label{coro:v_univer}
\end{corollary}

\section{Solving the Classification Problem for Split-Comparability Graphs}\label{section:proof}

In this section we present the main result, Theorem~\ref{thm:principal}, stated below.

\begin{theorem}
A split-comparability graph $G$ is Class 2 if and only if $G$ is neighborhood-overfull. 
\label{thm:principal}
\end{theorem}

\begin{proof}
Clearly, when $G$ is neighborhood-overfull it must be Class 2. We may therefore assume that $G$ is not neighborhood-overfull. If $\Delta(G)$ is odd, by Theorem~\ref{thm:chenfuko}, $G$ is Class 1. Hence, we may assume $\Delta(G)$ is even.

We shall consider that $Q$ is maximum, which implies that every $\Delta(G)$-vertex belongs to $Q$. According to the definition of a split-comparability graph, every $\Delta(G)$-vertex will either be in $Q_l$ or $Q_r$. If only one of such sets ($Q_l$ and $Q_r$), say $Q_l$, has a  $\Delta(G)$-vertex, by Corollary~\ref{coro:v_univer} the subgraph induced by $Q \cup S_l\cup S_t$, which is the semicore of $G$, has a universal vertex. As $G$ is not neighborhood-overfull then its semicore is not overfull. By Theorem~\ref{thm:vert-universal}, since the semicore of $G$ has a universal vertex and is not overfull, then the semicore is Class 1. By Theorem~\ref{thm:seminucleo}, we conclude that $G$ is Class 1 as well. Thus, we may consider that both $Q_l$ and $Q_r$ have $\Delta(G)$-vertices. By Proposition~\ref{prop:qlsl}, every $\Delta(G)$-vertex of $Q_l$ is adjacent to every vertex of $S_l \cup S_t$ and every $\Delta(G)$-vertex of $Q_r$ is adjacent to every vertex of $S_r \cup S_t$. Since both  $Q_l$ and $Q_r$ have $\Delta(G)$-vertices we deduce that $\Delta(G) = |Q| - 1 + |S_l| + |S_t| = |Q| - 1 + |S_r| + |S_t|$; whence $|S_l| = |S_r|$. 

By Theorem~\ref{thm:sheila}, we shall consider that every vertex in $S$ has degree either smaller than $|Q|/2$ or greater than $\Delta(G) / 2$. By definition, $Q_r \cap Q_l = \emptyset$, so $Q_r$ and $Q_l$ cannot have cardinality greater than $\Delta(G) / 2$, once $\Delta(G) > |Q|$. If $Q_r$ and $Q_l$ have cardinality smaller than $|Q|/2$ then $Q_t \neq \emptyset$. Moreover,  $|Q_t| + |Q_r| = |Q| - |Q_l| > |Q|/2$. In this case, let $v \in S_r$ be a vertex with degree $|Q_r|$ and let $G'$ be the graph obtained from $G$ by the addition of edges between $v$ and vertices of $Q_t$ so as to ensure that $|Q|/2 \leq d_{G'}(v) \leq \Delta(G)/2$. It is easy to see that $\Delta(G') = \Delta(G)$. By Theorem~\ref{thm:sheila}, $G'$ is Class 1. Hence $G$, which is a subgraph of $G'$, is Class 1 as well. We may thus assume, without loss of generality, that $|Q_r| < |Q|/2$ and $|Q_l| > \Delta(G) / 2$. In sum, we shall consider the case where $\Delta(G)$ is even,  $|S_l| = |S_r|$, $|Q_l| > \Delta(G) / 2$ and $|Q_r| < |Q|/2$.

We label the vertices of $Q$ as in Theorem~\ref{thm:split-comp}, i.~e., ($v_1, v_2, \ldots , v_r$). Similarly, we label the vertices of $S_l$ as ($u_1, u_2, \ldots, u_s$) and of $S_r$ as ($w_1, w_2, \ldots, w_s$), where $s = |S_l| = |S_r|$. Finally, we label the vertices of $S_t$ as ($x_1,x_2,\ldots,x_{t}$), where $t = |S_t|$. Let $v \in Q_l$ be a $\Delta(G)$-vertex. In order to color the edges of $G$ we will construct a saturated graph $G_l$. Such graph will initially be defined as $G[N[v]]$, then some extra edges will be added to ensure it is saturated. We label the vertices  of $G_l$ with the same labels used in $G$. Note that, by Proposition~\ref{prop:qlsl}, $V(G_l) = Q \cup S_l \cup S_t$. While $G_l$ is not saturated, the following types of edges will be added, respecting the following order for addition: group 1 $= \{$ $(u_i,u_j)$, $1 \leq i < j \leq s \}$; group 2 $= \{$ $(u_i,v_j)$, $1 \leq j < q$, $1 \leq i \leq s$ e $(u_i,v_j) \not\in E(G) \}$; group 3 $= \{$ $(x_i,v_j)$, $1 \leq i \leq t$, $1 \leq j \leq r$ e $(x_i,v_j) \not\in E(G) \}$; group 4 $= \{$ $(u_i,v_j)$, $1 \leq i \leq s$ e $q \leq j \leq r \}$. Within each group of edges, we follow an order for such additions so as to make sure that an edge incident to a vertex $u_{i+1}$ (or $x_{i+1}$) will only be added after every possible edge incident to $u_i$ (or $x_{i}$) has already been added. 

The graph $G_l$ thus obtained has a universal vertex, an odd number of vertices and is saturated; therefore, by Theorem~\ref{thm:vert-universal}, $G_l$ is Class 1. By Lemma~\ref{lem:colo_equi}, $G_l$ has a balanced edge coloring with $\Delta(G)$ colors. Let $C = \{1, 2, \ldots, \Delta(G)\}$ be a set of colors and $\beta:E(G_l) \rightarrow C$ be a balanced edge coloring of $G_l$. We will now obtain an edge coloring $\alpha:E(G) \rightarrow C$ for $G$ from $\beta$ by first assigning: $\alpha(v_i,v_j) = \beta(v_i,v_j)$, $1 \leq i < j \leq r$; $\alpha(v_i,u_j) = \beta(v_i,u_j)$, $1 \leq i \leq p$ and $1 \leq j \leq s$, if $(v_i,u_j) \in E(G)$; $\alpha(v_i,x_j) = \beta(v_i,x_j)$, $1 \leq i \leq r$ and $1 \leq j \leq t$, if $(v_i,x_j) \in E(G)$; $\alpha(v_i,w_j) = \beta(v_i,u_j)$, $q \leq i \leq r$ and $1 \leq j \leq s$, if $(v_i,w_j) \in E(G)$. Note that all edges of $G$ that were not colored, must be incident to a vertex of $S_r$. We next show how to assign colors to such edges so as to obtain a $\Delta(G)$-edge-coloring for $G$. 

Consider the set $\{v_q, v_{q+1}, \ldots, v_r\}$ in $G_l$. We will abuse notation and call this set  $Q_r$ in $G_l$ too. There might exist a vertex $u_h \in G_l$ such that $u_h$ is adjacent to a non-empty and proper subset of $Q_r$. If there is such a vertex $u_h$, then it is unique, given the order specified for the addition of edges to $G_l$ in Group 4. Assume $u_h$ exists and let $d$ be the number of neighbors of $u_h$ in $Q_r$. Since $|Q_r| < |Q|/2 < \Delta(G)/2$ and every vertex not adjacent to $u_h$ is in $Q_r$, then the degree of $u_h$ in $\overline{G_l}$ is less than $\Delta(G)/2$. Since $G_l$ is saturated, we conclude that there must be a vertex $u_{h+1}$ not adjacent to any vertex of $Q_r$. As $|S_l| = |S_r|$, at least one vertex $w_{h+1}$ in $S_r$ has no colored edges. Since the coloring $\beta$ is balanced and $G_l$ is saturated, by Proposition~\ref{prop:satur_sobra}, every color misses exactly one vertex of $G_l$. Let $\ell(v)$ be the set of colors missed in vertex $v$ in edge coloring $\beta$. Observe that $\Delta(G) = d_{G_l}(v_i) + |\ell(v_i)|$, for every vertex $v_i \in Q_r$. Therefore, for every vertex $v_i \in Q_r$  there are enough colors in $\ell(v_i)$ to color the uncolored edges of $G$ incident to $v_i$.

Let $C'$ be the set of colors incident to $w_h$ in coloring $\alpha$. Note that $\alpha(v_i,w_h) = \beta(v_i,u_h)$, for $q \leq i < q + d $. For every $v \in Q_r$ where $\ell(v)$ has a color $c \in C'$, if $(v,w_{h+1}) \in E(G)$ define $\alpha(v,w_{h+1}) = c$ and remove the color $c$ from $\ell(v)$ (even if no such edge exists). We can see $\ell(v)$ has enough colors to color the remaining edges incident to $v$, since if $(v,w_{h+1}) \not\in E(G)$ then $v$ is a not $\Delta(G)$-vertex in $G$. Note that the coloring $\alpha$ obtained so far is an edge coloring because $w_{h+1} \neq w_h$. 

Now we have only some edges $(v_i,w_j)$ to color, where $v_i \in Q_r$ e $w_j \in S_r$. By construction, $\ell(v_i) \cap C' = \emptyset$ and $\ell(v_i) \cap \ell(v_k) = \emptyset$, for any two distinct vertices $v_i$ and $v_k$ of $Q_r$. Therefore, we can assign the remaining colors of $\ell(v_i)$ arbitrarily for the uncolored edges incident to $v_i$ in order to obtain an edge coloring with $\Delta(G)$ colors. So, $G$ is Class 1.
\end{proof}

By Lemma~\ref{lem:compover}, every split-comparability graph satisfies the hypothesis of the Overfull Conjecture. 

\begin{lemma}
Every split-comparability graph $G$ has $\Delta(G) > n/3$.
\label{lem:compover}
\end{lemma}

\begin{proof}
Let $G$ be a split-comparability graph. Recall that we assumed that both $Q$ and $S$ are not empty, therefore, $n \geq 2$. If $|Q| = 1$ then $G$ has a universal vertex, so $\Delta(G) = n-1$ and $\Delta(G) >  n/3$ when $n \geq 2$. We may thus assume $|Q| \geq 2$. Adjust notation, if necessary, so that $|S_l| \geq |S_r|$. By Proposition~\ref{prop:qlsl}, $v_1$ is adjacent to every vertex in $S_l$ and $S_t$. So $d_Q  = |S_l| + |S_t|$. In addition, $\Delta(G) = |Q| - 1 + d_Q$. We therefore deduce that $\Delta(G)  =  |Q| - 1 + |S_l| + |S_t| = |Q| + |S| - |S_r| - 1 = n - |S_r| - 1$.

Assume the contrary, that is, $\Delta(G) \leq n/3$, or alternatively, $\Delta(G) - n/3 \leq 0$. Therefore, $\Delta(G) - n / 3 = n - |S_r| - 1 - n/3 = 2n/3 - 1 - |S_r| \le 0$.
We conclude that $|S_r|  \geq 2n/3 - 1$. Since $|S_l| \geq |S_r|$, thus $|S_l| + |S_r| \geq  4n/3 - 2 > n - 2$. In other words, if $\Delta(G) \leq n/3$, as  $|Q| \geq 2$, this implies that $n = |Q|+|S| > 2 + n - 2$, a contradiction.
\end{proof}

\section{Conclusion}

We proved that every split-comparability graph is Class 2 if and only if it is neighborhood-overfull. There are important implications of this result. First, the Overfull Conjecture holds for this class, by Lemma~\ref{lem:compover} and Theorem~\ref{thm:principal}. The Conjecture of Figueiredo, Meidanis and Mello, which states that a chordal graph is Class 2 if and only if it is neighborhood-overfull, also holds for this subclass of chordal graphs. Finally, since a neighborhood-overfull graph can be recognized in polynomial time, our proof implies that the Classification Problem can be solved in polynomial time for split-comparability graphs.

\footnotesize
\section{Acknowledgements}
This research was supported by \textsc{fapesp}, \textsc{Funda\c{c}\~{a}o Arauc\'{a}ria}, \textsc{capes} and \textsc{cnp}q Brazilian Funding Agencies.

\bibliography{mybibfile}

\end{document}